\documentclass{amsart}
\usepackage{amscd,graphicx}

\newtheorem{theorem}[equation]{Theorem}
\newtheorem{lemma}[equation]{Lemma}
\newtheorem{proposition}[equation]{Proposition}
\newtheorem{corollary}[equation]{Corollary}
\theoremstyle{definition}
\newtheorem{definition}[equation]{Definition}

\newtheorem{example}[equation]{Example}
\newtheorem*{remark}{Remark}
\DeclareMathOperator{\coker}{coker}

\DeclareMathOperator{\st}{st}
\DeclareMathOperator{\rk}{rank}
\newcommand{\Z}{\mathbf{Z}}

\newcommand{\KK}{\mathbf{K}}
\begin{document}

\title[Topological Koszulity]{The Koszul property as a Topological Invariant and Measure of Singularities.}
\author[Sadofsky, Shelton]{Hal Sadofsky and Brad Shelton}
\address{University of Oregon\\
Eugene Oregon}
\email{sadofsky@uoregon.edu, shelton@uoregon.edu}
\begin{abstract}
Cassidy, Phan and Shelton associate to any regular cell complex $X$ a
quadratic $K$-algebra $R(X)$.  They give a combinatorial solution to
the question of when this algebra is Koszul.  The algebra $R(X)$ is a
combinatorial invariant but not a topological invariant.  We show that
nevertheless, the property that $R
(x)$ be Koszul is a topological invariant.

In the process we establish some conditions on the types of local
singularities that can occur in cell complexes $X$ such that $R
(X)$ is Koszul, and more generally in cell complexes that are pure and
connected by codimension one faces.
\end{abstract}
\date{November 4, 2009}

\maketitle

\section{Introduction}

Let $X$ be a finite {\it regular} cell complex of dimension $d$ and
let $\KK$ be a field.  Following \cite{CPS1}, we will associate to
$X$, under certain global assumptions, a quadratic $K$-algebra $R(X)$
(defined below).  The main focus of \cite{CPS1} is to determine the
combinatorial properties required for this algebra to be {\it Koszul}.
The primary focus of this paper is to show that the Koszul property is
actually a topological invariant, even though the algebra is not.  In
the process we see that our global assumptions also imply some
restrictions on singularities of appropriate spaces $X$.

After a definition of our two technical assumptions we can state our
main theorem.  Our complexes will be finite throughout.  We will not
generally restate this hypothesis.

\begin{definition} Let $X$ be a regular cell complex of dimension $d$.  
\item(1) $X$ is {\it pure} if $X$ is the closure of its open $d$-cells.
\item(2) A pure, finite regular cell complex $X$, is {\it connected through codimension one faces} if the space $X- X^{(d-2)}$ is path connected (where $X^{(d-2)}$ is the $(d-2)$-skeleton of $X$). 
\end{definition}

\begin{theorem}\label{main-theorem-intro}
Let $X$  be a pure regular cell complex of dimension $d$, connected through codimension one faces.  Then $R(X)$
is Koszul (for the field $\mathbf{K}$) if any only if the following conditions both
hold.
\begin{enumerate}
\item $\tilde{H}_{i}(X;\KK) = 0$ for $i<d$.
\item $\tilde{H}_{i}(X,X-\{p \};\KK) = 0$ for each $p \in X$ and
each $i<d$.
\end{enumerate}
\end{theorem}

Because our hypotheses on the cell complex structure and on homology
are obviously homeomorphism invariant,
Theorem~\ref{main-theorem-intro} shows that the Koszul property for $R
(X)$ is a homeomorphism invariant.
We point out, however, that one does not have any nice homotopy invariance.

\begin{example}\label{koszul-not-top-invt}
There are homotopy equivalent, pure regular cell complexes $X$ and $Y$
of dimension $3$ such that $R (X)$ is not Koszul but $R
(Y)$ is Koszul.
\end{example}

Take $Y$ to be the union of two $3$-cells attached by some
2-dimensional face.  Then $Y$ satisfies the hypotheses of
Theorem~\ref{main-theorem-intro}.  Take $X$ to be the Example
5.9 of \cite{CPS1}, which is described explicitly as Example~\ref{s1-bad}.  $X$ is homotopy equivalent to $Y$ since one gets
a space homeomorphic to $Y$ by collapsing the contractible subcomplex
$a$ of $X$ to a point.  But \cite{CPS1} shows $R (X)$ is not Koszul (as one can also see by Theorem~\ref{main-theorem-intro}).

Although our argument does not make direct use of the definition of $R
(X)$, we review that definition here in the interest of
self-containedness.

Let $P$ be any finite ranked poset with minimal element $\bar 0$. For
each $x\in P$ let $s_{1}(x) = \{y\in P\,|\, y<x, \rk(x)-\rk(y)=1\}$ (the
elements immediately below $x$ in the ranked poset).  We define $R(P)$ to be the quadratic $\KK$-algebra on generators $r_x$, $x\in P-\{\bar 0\}$ with relations:
\[
r_x r_y = 0 \hbox{ for all } y\not\in s_{1}(x)
\]
and
\[
r_x \sum\limits_{z\in s_{1}(x)} r_z = 0 \hbox{ for all } x
\]

The set of all {\it closed} cells of a regular cell complex, together
with the empty set, form a finite ranked poset under set inclusion.
Following \cite{CPS1}, this is denoted $\bar P(X)$.  (In this poset
the rank of a cell is one more than its dimension, so that the rank of
the empty set is $0$).

If we assume that $X$ is pure, then we may adjoin one additional
(maximal) element to the poset $\bar P(X)$.  The resulting poset,
which is denoted $\hat P(X)$ is still a ranked poset.  If we further
assume that $X$ is connected through codimension one faces, then $\hat
P(X)$ has the combinatorial property known as {\it uniform}
(cf. \cite{GRSW05}).  Then we define $R(X)$ to be $R(\hat P(X))$.
While $R (\bar{P}(X))$ is always Koszul under the hypotheses that $X$
is pure and connected through codimension one faces (see
\cite{CPS1} and \cite{RSW08}), the Koszul property for $R (X)$ is
substantially more subtle.  Theorem 5.3 of \cite{CPS1} gives a precise
statement in terms of the combinatorial cell structure of $X$
describing when $R(X)$ is Koszul (refer to Theorem \ref{CPS-main}
below).

\section{CPS cohomology and local homology}\label{local-CPS}

We fix $X$, a finite regular CW complex of dimension $d$.  We begin by recalling the definitions of the groups
$H_{X} (n,k)$
from \cite[\S 4]{CPS1}, which we will write as $H^{n}_{k}(X)$. 

Assign orientations to each cell of $X$.
If $\beta$ is an $n$ cell and $\alpha$ is an $n+1$ cell, let
$[\alpha  : \beta]$ be the incidence number of $\beta$ in $\alpha$.
Because $X$ is regular, this is either $0, 1$ or $-1$. These
incidence numbers are usually defined in the context of cellular
homology so that, if $C_{*}(X)$ is the cellular chain complex of $X$,
and $\alpha$ is an $n+1$ cell, 
\[
d (\alpha) = \sum_{n-1 \mbox{ cells } \beta}[\alpha :\beta]\beta 
\]
Because $X$ is finite, and because we have a chosen basis for the
cellular chains (given by the cells of $X$) we have an isomorphism
between the cellular chains and the cellular cochains of $X$.  We
consider the cochains in dimension $n$ to be generated by the basis
dual to the $n$ cells of $X$, but we will use the same notation.  That is, an
$n$-cell $\alpha$ when considered as a generator of $C^{n}(X)$ will be
thought of as dual to the $n$-cell $\alpha$ in the basis of $C_{n}(X)$
provided by the $n$-cells.  With this identification, the coboundary map $\delta 
:C^{n}(X) \rightarrow C^{n+1}(X)$ is given by 
\[
\delta  (\alpha) = \sum_{n+1 \mbox{ cells } \beta}[\beta :\alpha]\beta.
\]
We define $C^{n}_{k}(X)$ to be the submodule of $C^{n}(X)\otimes
C_{k}(X)$ generated by $\alpha \otimes \beta$ such that $\beta
\subseteq \partial \alpha$ (that is, the cell associated to $\beta$ is
a subset of the boundary of the cell associated to $\alpha$). Then $d$
induces a differential $C^{n}_{k}(X) \rightarrow C^{n}_{k-1}(X)$ and
$\delta$ induces a differential $C^{n}_{k}(X) \rightarrow
C^{n+1}_{k}(X)$.

\begin{definition}\label{l-n-k}
(\cite[Definition 4.1]{CPS1}) For each $k$ and $n$, let:
\[
L^{n}_{k}(X) = \coker (C^{n}_{k+1}(X) \xrightarrow{d}C^{n}_{k}(X) ).
\]
Then $L_k^*(X)$ is a cochain complex with differential induced by $\delta$.
The CPS cohomology groups of $X$ are defined by 
\[
H^{n}_{k}(X) = H^{n}(L^{*}_{k}(X)).
\]
These cohomology groups are defined with coefficients in $\Z$.  We write $H_k^n(X;R)$ to denote the same groups when calculated with coefficients in a commutative coefficient ring $R$.
\end{definition}

We now recall Theorem 5.3 of \cite{CPS1}.

\begin{theorem}\label{CPS-main} Let $X$ be a pure regular cell complex of dimension $d$, connected by codimension one faces. Then the $\KK$-algebra $R(X)$ is Koszul if and only
if $H^{n}_{k}(X,\KK) = 0$ for $0\leq k<n<d$.
\end{theorem}

For our purposes it is convenient to present a reformulation of Theorem \ref{CPS-main} in terms of relative cohomology groups involving the {\it stars} of the cells of $X$.

\begin{definition}
The \emph{star} of a cell $\sigma$ in a regular cell complex $X$ is
\[
\st (\sigma) = \{y\in X: y \text{ is in some open cell whose closure contains }
\sigma \}
\]
\end{definition}
We note that $\st (\sigma)$ is an open subset of $X$.  We also use
$\st^{l}(\sigma)$ to denote the union of the open cell $\sigma$ with
all open cells in $\st (\sigma)$ of dimension $\leq l$.

\begin{theorem}\label{CPS-reformulated} Let $X$ be a pure regular cell complex of dimension $d$, connected by codimension one faces. Then the $\KK$-algebra $R(X)$ is Koszul if and only 
\item(1) $\tilde H^n(X,\KK) = 0$ for $n<d$ and 
\item(2) For every $k$-cell $\sigma$ and $k+1<n<d$, $H^n(X,X-\st(\sigma);\KK) = 0$. 
\end{theorem}

\begin{remark} Theorem \ref{CPS-reformulated} is a reformulation of Corollary 5.8 in \cite{CPS1}.  We wish to point out that the condition $k+1<n$ was inadvertently omitted in their statement. 
\end{remark}

As we will see in the next section, the cohomology groups $H^n(X,X-st(\sigma))$ can be replaced by the {\it local homology} groups 
$H_n(X,X-\{x\})$ for any $x\in \sigma$ (see Lemma \ref{def-retract}).   This suggests the following definition.

\begin{definition}\label{singular-n}
We define the set $S_{n}$ (relative to the ring of coefficients $R$) by $x \in S_{n}$ if $H_{i} (X,X-\{x \};R) = 0$
for $i \leq  n$ and $H_{n+1} (X, X-\{x \};R) \ne 0$.
\end{definition}

Now we can state and prove a proposition equivalent to Theorem \ref{main-theorem-intro}.  We will leave certain technical aspects of the proof to the subsequent two sections, as well as a more extensive discussion of the structure and significance of the sets 
$S_n$.

\begin{proposition}\label{main-prop}
Let $X$ be a pure regular cell complex of dimension d which is
connected through codimension one faces.  Then $R(X)$ is Koszul if and only if
\item(1) $\tilde H^n(X,\KK) = 0$ for $n<d$ and 
\item(2) The sets $S_k$ (relative to $\KK$) are empty for $0\le k \le d-2$. 
\end{proposition}

\begin{proof}
In this proof all homology and cohomology groups should be computed relative to the field $\KK$.  We suppress this from the notation.

We need only see that condition (2) of Theorem \ref{CPS-reformulated} and condition (2) of Proposition \ref{main-prop} are equivalent under the hypotheses on $X$ and condition (1).  Suppose the sets $S_k$ are empty for $0\le i\le d-2$.  Then by
Lemma~\ref{def-retract} $H_{i}(X,X-\st (\sigma)) = 0$ for every cell
$\sigma$ and every $i<d$.  The same follows by the universal
coefficient theorem for $H^{i}(X,X-\st (\sigma))$.

Conversely, assume that for some $i \le d-2$, the set $S_i$ is not empty. Let $m$ be minimal such that
$S_{m} \ne \emptyset$.  By Lemma~\ref{singular-set-union-of-cells} and Proposition~\ref{singularity-is-very-singular}, $S_m$ is a union of cells and must contain a cell $\alpha$ of dimension $k< m$.   By Lemma \ref{def-retract}, $H^{m+1}(X,X-st(\alpha))$ does not vanish, contradicting (2) of Theorem \ref{CPS-reformulated}.
\end{proof}

\section{Preliminary homotopy results}

Let $X$ be a regular cell complex of dimension $d$.  If $x \in X$, we write $\sigma (x)$ for the unique open cell of $X$
containing $x$.  
The following is a standard lemma of piecewise linear
topology.

\begin{lemma}\label{def-retract}
Given a cell $\sigma$, $\st (\sigma)$ is contractible (and in fact has
a strong deformation retract to $\sigma$).  Also, given any point $x
\in \sigma $, there is a strong deformation retract
\[
X- \{x \} \rightarrow X- \st (\sigma (x)).
\]
\end{lemma}
\begin{proof}
To see $\st (\sigma)$ has a strong deformation retract to $\sigma $ we
want a homotopy 
\[
H: \st (\sigma)\times I \rightarrow  \st (\sigma ).
\]
Of course $H|_{\sigma \times I}$ will just be the projection to
$\sigma$.  

Now suppose $H$ has been defined on the subset of $\st (\sigma)$
consisting of $\sigma$ together with other open cells up through cells
of dimension $l$.  Since $X$ was a regular cell complex, for each open
$l+1$ cell $E$ of $\st (\sigma )$, $H$ is defined on a contractible
subset of the boundary, $E' \subseteq \partial E$.  So $H$ is defined
on $W = E\times \{0 \}\cup E'\times I$.  The pair $(E\times I,W)$
has the homotopy extension property (see \cite[p. 23]{Hatcher02}), so
we use that to define $H$ on $E\times I$.

To define our retract   
\[
X- \{x \} \rightarrow X- \st (\sigma (x)).
\]
we begin by noting there is a strong deformation retract
$\overline{\sigma (x)} - \{x \}$
to $\partial \overline{\sigma (x)}$.
Now assume the homotopy is defined on $\st^{l}(\sigma (x)) -\{x \}$
(and of course is the identity on $X-\st (\sigma (x))$).  Note that
$\st^{0}(\sigma (x)) = \sigma (x)$.

Let $E$ be the closure of an $l+1$ cell of $\st (\sigma (x))$.  Since the pair
\[
( (E-\{x \})\times I,(E-\{x \})\times \{0  \}\cup (\partial E - \{x
\})\times I)
\]
has the homotopy extension property, extend $H$ across $E-\{x \}$.
Continue until the homotopy is defined on all of $X - \{x \}$.
\end{proof}

\begin{corollary}\label{main-retract}
Given a cell $\sigma$ there is a strong deformation retract 
\[
X-\sigma
\rightarrow X- \st (\sigma)
\]
such that if $x \in E$ for some cell $E$, then the image of $x\times
I$ is in $\overline{E}$, and meets no cells of $\st (\sigma)$ other
than $E$.
\end{corollary}
\begin{proof}
We apply the strong deformation retract of Lemma~\ref{def-retract} to the space $X- \sigma$.
\end{proof}

As an application of Corollary~\ref{main-retract} we have the
following.

\begin{corollary}\label{point-retract}
Let $D$ be an open $n$-cell of $X$  and $\sigma $  a $0$-cell in $\partial D$.
Then 
\[
X - (\sigma  \cup D)  \simeq X-\sigma.
\]
\end{corollary}
\begin{proof}
Apply the homotopy from Corollary~\ref{main-retract} to the space $X- (\sigma  \cup
D)$.  This gives a retract of $X- (\sigma  \cup D)$ to $X- (\st (\sigma ))$, and
since that space is also a retract of $X-\sigma $, we get $X- (\sigma  \cup D)
\simeq X-\sigma$.
\end{proof}

\begin{proposition}\label{subdivision-retract}
Let $X$ be the realization of a simplicial complex $\Delta$, and let $A$
be a closed $i$-simplex in $\Delta'$ (the first barycentric subdivision of $\Delta$).
Let $v$ be the vertex that $A$ shares with an $i$-simplex of $\Delta$.
Then 
\[
X - \{v \} \simeq X-A.
\]
\end{proposition}
\begin{proof}
In the complex given by $\Delta$ we can construct the deformation retract
of $X - \{v \}$ to $X - \st (v)$  (where $\st (v)$ is defined using
the simplicial complex $\Delta$)
\[
H: ( X- \{v \})\times I \rightarrow X- \{v \}.
\]
explicitly by using barycentric coordinates in each simplex of $\Delta$.

Specifically,  if $\sigma$ is a simplex of $\Delta$ not containing $v$,
then $H (p,t) = p$ for $p \in \sigma$.  

If $\sigma$ does contain $v$, let the vertices of $\sigma$ be
$v=v_{0},v_{1},\dotsc ,v_{k}$. Then a typical point of 
$\sigma -\{v\}$ is given by 
$sv_{0} + (1-s) \sum_{i=1}^{k}a_{i}v_{i}$ where
$\sum_{i=1}^{k}a_{i} = 1$.  Then 
\[
H (sv_{0} + (1-s) \sum_{i=1}^{k}a_{i}v_{i},t) = (1-t) sv_{0} +
(1-s+ts) \sum_{i=1}^{k}a_{i}v_{i}
\]

Applying this homotopy to $X - A$ gives a deformation retract to $X -
\st (v)$.  So $X- \{v \} \simeq X- A$. 
\end{proof}

\section{Singularities detected by local homology}\label{local-vanishing}

\subsection{The singular sets $S_{n}$ are composed of cells of
dimension less than or equal to $n$.}

We continue to assume that $X$ is a finite regular cell
complex of dimension $d$. Throughout this section we assume further that $X$ is pure.  Recall that we refer to $H_*(X,X-x)$ as the local homology at $x$.  Since $X$ is locally contractible we can choose a contractible neighborhood of $x$, say $U$. Then by excision we have 
\[
H_{*} (X,X-\{x \}) \cong  H_{*} (U,U-\{x \}) \cong \tilde{H}_{*-1} (U-\{x \}).
\]

From this we see that any $x$ in the interior of a $d$-cell of $X$ has
local homology $H_*(X,X-x) \cong \tilde H(S^d)$ and $x\in S_{d-1}$.
Similarly, if $x$ is on the boundary of exactly one $d$-cell then
$H_*(X,X-x) = 0$ and $x$ is none of the sets $S_k$.  So if we think of
$X$ as a singular manifold with boundary, the point with neighborhoods
homeomorphic to $\mathbf{R}^{d}$ or the corresponding half-space are
not in $S_{k}$ when $k<d-1$.  The sets
$S_i$, $0\le i \le d-2$ form a stratification of those singularities
of $X$ that are detected by local homology.

Of course it is also possible for $X$ to be topologically singular and
still have
local homology zero in dimensions below $d$ at every point.  A simple but
illustrative example (for $d=1$) is the space 
\[
( [0,1]\times \{a,b,c \})/ \{(0,a)  \sim (0,b) \sim (0,c)  \}.
\]
This is three copies of the unit interval identified at one end
point.  The identification point is a singular point and has no local
homology below dimension $1$.  This singularity is still
detected by local homology of course, but not until dimension $1$.

As is well known, there are also spaces with singularities so that all the
local homology groups are those of a manifold.  A standard source of
examples is the suspension of any homology sphere which isn't actually
a sphere.

We begin by showing that the sets $S_n$ put restrictions on the cell
structure of $X$.  Recall first (Definition~\ref{singular-n}) that
$S_{n}$ does not depend on the cellular structure of $X.$
Nevertheless, we have the following.

\begin{lemma}\label{singular-set-union-of-cells}
The set $S_{n}$ is a union of open cells (in any cell structure on $X$).
\end{lemma}
\begin{proof}
If $x$ is in some open cell
$D$ then $X-D \simeq X- \{x \}$ by 
Lemma~\ref{def-retract} together with Corollary~\ref{main-retract}.
So applying the same argument to $x' \in D$ and using the appropriate
long exact sequences,
\[
H_{*}(X, X-\{x \}) \cong H_{*}(X,X-D) \cong H_{*}(X,X-\{x' \})
\]
\end{proof}

\begin{lemma}\label{singularity-level}
If $x \in S_{n}$ for $n<d-1$ then $x$ must be in the interior of a cell of
dimension $n$ or lower.
\end{lemma}

Note that this fact depends on $X$ being pure.  For example, If we take $X$ to be the union of a two cell
and a one cell at a vertex, then points in the interior of the one
cell will be in $S_{0}$.   Geometrically, $x \in S_{n}$ says that if
we take a contractible neighborhood of $x$ and remove $x$ from that
neighborhood then the resulting set is no longer $n$-connected.

\begin{proof}
Recall $st^{k}(\sigma)$ is the union of $\sigma $ and the open cells of dimension
$k$ and lower which are contained in $\st (\sigma)$.  This is the same
as $\st
(\sigma)$ within the space $X^{(k)}$ if $\sigma$ is a cell of
dimension less than or equal to $k$.

Suppose $x$ is in the interior of a cell of dimension $k<d$. We have a
commutative square of spaces
\[
\CD
X^{(k+1)}- \{x \} @>>> X^{(k+1)}-\st^{(k+1)} (\sigma (x))\\
@VVV @VVV\\
X - \{x \} @>>> X - \st (\sigma (x)).
\endCD
\]
The horizontal maps are homotopy equivalences by
Lemma~\ref{def-retract}.  The spaces on the right are subcomplexes of
$X$ and the right hand vertical map is inclusion of
the  $k+1$-skeleton.  So by cellular approximation, all maps induce
isomorphisms in $H_{i}$ for $i<k+1$.

From the long exact sequence of a pair, it follows that 
\[
H_{i} (X^{(k+1)},X^{(k+1)}- \{x \}) \rightarrow H_{i} (X, X-\{x \})
\]
is an isomorphism for $i\leq k$.  Now let $U = \st^{(k+1)} (\sigma
(x))$, which is an open neighborhood of $x$ in $X^{(k+1)}$. $U$
consists of the open $k$-cell containing $x$ and any open $k+1$ cells
which have that $k$-cell as a face.  So $U$ looks like a finite
collection of $k+1$-cells identified along part of their boundary, and
$x$ is in that part of the common boundary.  It follows that $U-\{x
\}$ is homotopy equivalent to a wedge of $k$-spheres (one fewer than
the number of $k+1$-cells attached to $\sigma (x)$).

So 
\[
H_{i} (X^{(k+1)},X^{(k+1)}- \{x \}) = H_{i} (U,U-\{x \})
\]
is $0$ for $i<k+1$ (and is free abelian on one fewer generator
than the number of $k+1$-cells attached to $\sigma (x)$ for $i=k+1$).

It follows that $x$ is not in $S_{n}$ for $n<k$.
\end{proof}

See the appendix for examples where $S_{n}$ contains the interiors of
cells of dimension strictly smaller than $n$.

\subsection{The implications of connectivity by codimension one faces}

We have already assumed the global topological condition: $X$ is pure. Our final goal is to understand the effect of the extra global topological condition: connected by codimension one faces.  Under that condition we can prove a remarkable strengthening of Lemma \ref{singularity-level} (see Proposition \ref{singularity-is-very-singular} and its Corollary). We require one technical lemma.



\begin{lemma}\label{boundary-subset-lemma} Let $X$ be a pure regular
cell complex of dimension $d$.  Let $n<d$.  Suppose $S_{0} = \dotsb =
S_{n-1} = \emptyset,$ $\tilde{H}_{k} (X) = 0$ for $k<d$, and $D$ is an open
$n$-cell of $S_{n}$ with $S_{n} \cap \partial D = \emptyset$. Let $Y=X-D$. 

Let $A \subseteq \partial D$ be a subspace homeomorphic to $D^{i}$
(the closed $i$ disk) and
also a subcomplex of $\partial D$ under some cell structure on
$\partial D$ which subdivides the given cell structure.

Then $\tilde{H}_{j}(Y-A) = 0$ for $j<n+1-i$.
\end{lemma}
\begin{proof}
The proof is by a double induction with the outer induction on $i$ and
the inner induction on the number of $i$-cells in $A$,
which we'll denote by $r$.

Let $i$ be $0$.  Note that $r=1$ by our hypotheses
that $A \cong D^{0}.$ Then by Corollary~\ref{point-retract}, $Y-A
= X- (A\cup D)  \simeq
X-A$.  Since $A$ is a single point in $\partial D$, and by
hypothesis that point isn't in $S_{0}\cup \dotsb \cup S_{n}$,  we have
$\tilde{H}_{j}(X-A) = 0$ for $j<n+1$ as desired.

Now suppose the lemma is established for $i-1 \geq 0$.  Consider first
the following special case.  Subdivide the cell complex structure on
$\partial D$ so that it is a simplicial complex.  Then take the first
barycentric subdivision of that simplicial complex.  Let $A$ be the
closure of an $i$-cell in that complex,  so $A \cong D^{i}.$

Let $v$ be the vertex that $A$ shares with the $i$-simplex (before
subdivision) that $A$ is part of.
By Proposition~\ref{subdivision-retract}, $Y-A \simeq Y- \{v \}$
which is in turn homotopy equivalent to $X - \{v \}$ by the previous
case.
So $\tilde{H}_{j}(Y-A) = 0$ for $j<n+1$.

Now let $A$ be as in the hypotheses, with the additional assumption
that it is a subcomplex of a barcyentric subdivision of a simplicial
subdivision of $\partial D$, as in the special case above.  Suppose
$A$ has $r+1$ cells, and that the lemma is true in the case of $r$
cells.  Write $A = A' \cup A''$ where $A'$ is a single cell, and $A''$
has $r$-cells, and $A' \cap A''$ is homeomorphic to $D^{i-1}$.

We look at the Mayer-Vietoris sequence for 
\[
Y- (A' \cap  A'') = (Y-A') \cup (Y-A'')
\]
which gives 
\begin{multline*}
H_{j+1} (Y-A')\oplus H_{j+1} (Y-A'') \rightarrow H_{j+1} (Y- (A'\cap
A'')) \rightarrow H_{j} (Y-A)\\
 \rightarrow H_{j} (Y-A') \oplus H_{j} (Y-A'').
\end{multline*}
By our two inductive hypotheses,  (on $i-1$ and $r$), $H_{j+1} (Y-
(A'\cap A'')) = 0$ for $j+1<n+1- (i-1)$ (or $j<n+1-i$) and $H_{j}
(Y-A'') = H_{j} (Y-A') = 0$ for $j<n+1-i$.

It follows that $H_{j} (Y-A) = 0$ for $j<n+1-i$ as we want.  By
induction on $r$ this holds for any $A$ which is an appropriate
subcomplex of our subdivision (of the cell structure on $\partial D$).

Finally, if $A \subseteq \partial D$ is any appropriate subcomplex of
a subdivision of $\partial D$ so that $A \cong D^{i}$, then $A$ is also an
appropriate subcomplex of a finer subdivision of $\partial D$ which is
itself a barycentric subdivision of a simplicial complex.  So our
special case covers this subcomplex $A$ of $\partial D$.
\end{proof}

\begin{proposition}\label{singularity-is-very-singular}
Let $X$ be a complex as above.  In addition assume that $\tilde{H}_{i} (X) =
0$ for $i<d$, and that $X$ is connected through codimension one
faces.  If there is an $n<d-1$ so that $S_{n} \ne \emptyset$, then
there is some point in some such $S_{n}$ which is in an open cell of
dimension smaller than $n$.
\end{proposition}
\begin{proof}
Let $n$ be minimal so that $S_{n} \ne \emptyset$.  If there is no such
$n$, or if $n\geq d-1$, we're done.  So assume $n<d-1$.  By Lemmas~\ref{singular-set-union-of-cells} and \ref{singularity-level} $S_n$ must contain an open cell $D$ of dimension at most $n$.  If $D$ has
dimension less than $n$, then we are done.  So assume $D$ has dimension $n$.
Let $Y = X-D$. From the hypothesis $\tilde H_k(X) = 0$ for $k<d$ we get
$H_{n}(Y) = H_{n+1}(X,Y) \ne 0$. 

We wish to prove that $S_{n}\cap \partial D \ne \emptyset.$
Choose a sequence of subsets $A^{i}$, $B^{i}$, $i = 0,\dotsc ,n-1$
subcomplexes of $\partial D$ (or of some subdivision) so that 
\begin{enumerate}
\item $A^{n-1} \cup B^{n-1} = \partial D \cong S^{n-1}$
\item $A^{i} \cup B^{i} \cong S^{i}$
\item $A^{i} \cap B^{i} = A^{i-1}\cup B^{i-1}$.
\end{enumerate}
Notice that $A^0$ and $B^0$ are distinct singleton sets.

Assume $S_{n} \cap \partial D = \emptyset$.  Consider 
\begin{equation}\label{level-zero}
Y = ( Y - A^{0}) \cup (Y-B^{0}).
\end{equation}
The space $Y- A^{0} \simeq X- A^{0}$ by Corollary~\ref{point-retract}, so
since the point of $A^{0}$ is not in $S_{0} \cup \dotsb \cup S_{n}$,
we get $\tilde{H}_{j} (Y-A^{0}) = 0$ for $j\leq n$, and of course the
same result for $\tilde{H}_{J} (Y-B^{0})$.

Then in the Mayer-Vietoris sequence for (\ref{level-zero}) we get 
\[
H_{n} (Y) \cong  H_{n-1} (Y- (A^{0}\cup B^{0})).
\]

We do a similar analysis for 
\begin{equation}\label{level-one}
Y- (A^{0}\cup B^{0}) = ( Y-A^{1}) \cup (Y-B^{1}).
\end{equation}
We have $\tilde{H}_{j} (Y-A^{1}) = 0$ for $j<n+1-1 = n$ by
Lemma~\ref{boundary-subset-lemma}.  

Then in the Mayer-Vietoris sequence for (\ref{level-one}) we get 
\[
H_{n-1} (Y- (A^{0}\cup B^{0})) \cong  H_{n-2} (Y- (A^{1}\cup B^{1})).
\]

Similarly we have 
\begin{equation}\label{level-k}
Y- (A^{k-1}\cup B^{k-1}) = (Y-A^{k})\cup (Y-B^{k}),
\end{equation}
Lemma~\ref{boundary-subset-lemma} tells us that $\tilde{H}_{j}
(Y-A^{k}) = 0$ for $j<n+1-k$  (and a similar result for $B^{k}$).

So by the Mayer-Vietoris sequence for (\ref{level-k}) we get 
\[
\tilde{H}_{n-k} (Y- (A^{k-1}\cup B^{k-1})) \cong  \tilde{H}_{n-k-1} (Y-
(A_{k}\cup B_{k})).
\]

Assembling this information, we get 
\[
0 \ne H_{n} (Y) = \tilde{H}_{0} (Y- (A^{n-1}\cup B^{n-1})) =
\tilde{H}_{0} (X-\overline{D}).
\]
The hypothesis that $X$ is connected through codimension one faces
tells us that $\tilde{H}_{0} (X-\overline{D}) = 0$ unless (possibly)
$D$ has codimension $1$.
But $D$ was assumed to have dimension $n< d-1$, so we have a
contradiction to our assumption that $S_{n} \cap \partial D = \emptyset$.
\end{proof}

\begin{corollary}\label{non-singular}
Suppose $X$ is a pure regular cell complex of dimension $d$, connected
through codimension one faces, and with $\tilde{H}_{i} (X) = 0$ for
$i<d$.

Then if for each $0\leq i<d-1$, $S_{i}$ contains no cells of dimension
less than $i$, then for $0 \leq  i < d-1$, each $S_{i}$ is empty.  
\end{corollary}

\appendix
\section{Examples of singularities}

As is clear from the above, $x \in S_{n}$ does not determine the
dimension of the open cell containing $x$.  For example, $S_{d-1}$
contains all open $d$ cells and all interior open $d-1$
cells. Similarly, the $3$-dimensional complex $X$ of
Example~\ref{s1-bad} below has points $x \in S_{1}$ so that $x$ is in
an open $1$-cell, and also at least one $x \in S_{1}$ which is in an
open $0$-cell.

Below we give examples of contractible cell complexes connected
through codimension one faces where (regardless of the chosen cell
structure) the singular set $S_{r}$ for some $r$ is composed of cells
of varying dimensions.

\begin{example}\label{s1-bad}
Let $T_{1}$ be a 3-simplex with vertices $\{v_{0},\dotsc ,v_{3} \}$
and $T_{2}$ be a 3-simplex with vertices $\{w_{0},\dotsc ,w_{3} \}$.
Define $X$ by 
\[
( T_{1} \sqcup T_{2})/\sim 
\]
where the relation $\sim$ is given by identifying the $2$-simplex
spanned by $\{v_{0},v_{1},v_{2} \}$ linearly (preserving the order of
the vertices) with that spanned by $\{w_{0},w_{1},w_{2}
\}$, and by identifying the $1$-simplex spanned by $\{v_{0},v_{3} \}$
with that spanned by $\{w_{0},w_{3} \}$ (again preserving
the order of simplices).  

If we just made the first identification we would have something
homeomorphic to a $3$-disk.  With both identifications, we have
something homotopy equivalent to a $3$-disk since it is a homotopy
equivalence to identify the contractible subcomplex consisting of the
closed 1-cell that is the image of the $1$-simplex spanned by
$\{v_{0},v_{3} \}$ to a point.

In this space, $S_{1}$ is the open 1-cell that is the image of the
open 1-simplex spanned by $\{v_{0},v_{3} \}$ (or equivalently by
$\{w_{0},w_{3} \}$) together with the image of $v_{0}$. The point that
is the image of $v_{0}$ will be a $0$-cell in any cell structure on
$X$, so that $S_{1}$ (in this example) will always have points
belonging to $0$-cells.
\end{example}

\begin{example}\label{s1-bad-4d} We can mimic Example~\ref{s1-bad} in
higher dimensions.  For example to create a space of dimension $4$
such that $S_{1}$ must necessarily contain both points of $1$-cells
and $0$-cells, we can define $X$ as follows.
Let $T_{1}$ be a 4-simplex with vertices $\{v_{0},\dotsc ,v_{4} \}$
and $T_{2}$ be a 4-simplex with vertices $\{w_{0},\dotsc ,w_{4} \}$.
\[
X = ( T_{1} \sqcup T_{2})/\sim 
\]
where the relation $\sim$ is given by identifying the $3$-simplex
spanned by $\{v_{0},v_{1},v_{2},v_{3} \}$ linearly (preserving the order of
the vertices) with that spanned by $\{w_{0},w_{1},w_{2},w_{3}
\}$, and by identifying the $1$-simplex spanned by $\{v_{0},v_{4} \}$
with that spanned by $\{w_{0},w_{4} \}$ (again preserving
the order of simplices).  

Then $S_{1}$ is analogous to the previous example;  it contains the
image of the $1$-simplex spanned by $\{v_{0},v_{4} \}$ together with
the image of $v_{0}$.  $S_{2} = \emptyset$.  As in the previous
example, the point that is the image of $v_{0}$ will be a $0$-cell in
any cell structure on $X$, and the rest of the points of $S_{1}$ will
be in $1$-cells or $0$-cells.
\end{example}

\begin{example}\label{s2-bad}
We can also creat a complicated $S_{2}$. 

Let $T_{1}$ be a 4-simplex with vertices $\{v_{0},\dotsc ,v_{4} \}$
and $T_{2}$ be a 4-simplex with vertices $\{w_{0},\dotsc ,w_{4} \}$.
\[
X = ( T_{1} \sqcup T_{2})/\sim 
\]
where the relation $\sim$ is given by identifying the $3$-simplex
spanned by $\{v_{0},v_{1},v_{2},v_{3} \}$ linearly (preserving the order of
the vertices) with that spanned by $\{w_{0},w_{1},w_{2},w_{3}
\}$, and by identifying the $2$-simplex spanned by $\{v_{0},v_{1},v_{4} \}$
with that spanned by $\{w_{0},w_{1},w_{4} \}$ (preserving
the order of simplices).  

Then $S_{1} = \emptyset$, but $S_{2}$ consisted of the open $2$-cell
that is the image of the simplex $\{v_{0},v_{1},v_{4} \}$ together
with the open $1$-cell that is the image of the simplex $\{v_{0},v_{1}
\}$.  This subset of $X$ will be a union of a non-zero number of open
$2$-cells and a non-zero number of open $1$-cells in any cell complex
on $X$.  
\end{example}

\begin{example}
It is also possible to create a complex $X$ of dimension 4 where
$S_{2}$ will be a union of a non-zero number of $2$-cells, a non-zero
number of $1$-cells and a non-zero number of $0$-cells for any cell
structure on $X$.
\begin{figure}[h!]
\includegraphics[scale=.45]{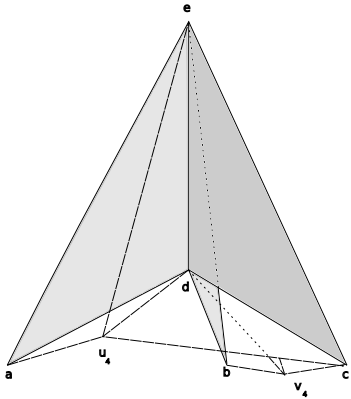}
\caption{$A$ with $\{u_{0},\dotsc ,u_{3} \}$ and $\{v_{0},\dotsc ,v_{3}
\}$ attached.}\label{subcomplex}
\end{figure}

We begin by creating the subcomplex most of which will become
$S_{2}$.  We will glue three $2$-simplices together along a common
edge.  So let $A$ be the simplicial complexes with vertices
$a,b,c,d,e$, 2-simplices $\{a,d,e \}, \{b,d,e \}, \{c,d,e \}$.
This determines the 1-simplices, and the 1-simplex common to the three
$2$-simplices is thus $\{d,e \}$.  This is illustrated by the three
shaded $2$-simplices in
Figure~\ref{subcomplex}. 

Next we attach three $4$-simplices to $A$ by attaching adjacent
$2$-faces (sharing a $1$-face) of each $4$-simplex to pairs of
$2$-simplexes in $A$.  Let $T_{1}$ be the 4-simplex with vertices
$\{u_{0},\dotsc ,u_{3} \}$, $T_{2}$ have vertices $\{v_{0},\dotsc
,v_{3} \}$ and $T_{3}$ have vertices $\{w_{0},\dotsc ,w_{3} \}$.  

We identify the 2-face spanned by $u_{0},u_{1},u_{2} $ with the
simplex $e,d,a $ and the 2-face spanned by $u_{0},u_{1},u_{3}$ with
the simplex $e,d,c $ (preserving the given order in both cases).

Similarly for $v_{0},v_{1},v_{2}$ with $e, d, c$ and then
$v_{0},v_{1},v_{3}$ with $e, d, b$.  And finally $w_{0},w_{1},w_{2}$
with $e, d, b$ and $w_{0},w_{1},w_{3}$ with $e, d, a$.  We sketch part
of this complex in Figure~\ref{subcomplex}, but note that we have
no realistic way to sketch the 4-simplices involved, so we are only
showing a (two dimensional sketch of a) three dimensional picture of
the space.

Of course this is contractible, and $S_{2}$ consists of the three open
$2$-simplices from $A$ together with the open $1$-simplex $e, d$.  But
this space is not connected through codimension one faces.  To fix
that we add a last 4-simplex with vertices $\{x_{0},\dotsc ,x_{4}\}$.
We identify the face with vertices $x_{0},\dotsc ,x_{3}$ with
$u_{1},u_{2},u_{3},u_{4}$, the face $x_{0}, x_{2}, x_{3}, x_{4}$ with
$v_{1}, v_{2}, v_{3}, v_{4}$, the face $x_{0}, x_{1}, x_{3}, x_{4}$
with $w_{1}, w_{3}, w_{2}, w_{4}$.

We'll call the resulting space $X$.  $X$ is now dimension 4, contractible and
connected through codimension one faces.  $S_{0} = S_{1} = \emptyset$
and $S_{2}$ is the union of the open $2$-cells of $A$ together with
the open $1$-cells $\{a,d \}, \{b,d \}, \{c, d \}$ and the $0$-cell
$\{d \}$.  In any cell structure on $X$, $\{d \}$ will be a zero cell,
and all but finitely many points of the open $1$-cells we just listed
will be in open $1$-cells, and of course almost all points in the
interiors of the $2$-cells listed will be in open $2$-cells.
\end{example}

\end{document}